\documentclass[final]{dmtcs-episciences}

\usepackage[utf8]{inputenc}
\usepackage{subfigure}

\usepackage[round]{natbib}
\usepackage{amsmath,amssymb, mathalpha, amsthm}
\newtheorem{theorem}{Theorem}[section]
\newtheorem{corollary}[theorem]{Corollary} 
\newtheorem{proposition}[theorem]{Proposition}

\theoremstyle{definition}
\newtheorem{definition}{Definition}
\newtheorem{example}{Example}
\theoremstyle{remark}
\newtheorem*{remark}{Remark} 

\newcommand\preceqdot{\mathrel{\ooalign{$\preceq$\cr
\hidewidth\raise0.225ex\hbox{$\cdot\mkern0.5mu$}\cr}}}
\DeclareMathOperator{\supp}{supp}
\DeclareMathOperator{\rev}{rev}
\DeclareMathOperator{\outdeg}{outdegree}

\author[Ryota Inagaki et al.]{Ryota Inagaki\affiliationmark{1}\thanks{Supported by the MIT Department of Mathematics.}
  \and Tanya Khovanova\affiliationmark{1}\thanks{Supported by the MIT Department of Mathematics.}
  \and Austin Luo\affiliationmark{2}}
\title{Permutation-based Strategies for Labeled Chip-Firing on $k$-ary Trees}
\affiliation{
  Massachusetts Institute of Technology, Cambridge, USA\\
  Morgantown High School, Morgantown, USA}

\begin{document}

\publicationdata{vol. 28:2}{2026}{2}{10.46298/dmtcs.16108}{2025-07-24; 2025-07-24; 2025-12-05}{2025-12-19}

\maketitle
\begin{abstract}
  Chip-firing is a combinatorial game on a graph, in which chips are placed and dispersed among its vertices until a stable configuration is achieved. We specifically study a chip-firing variant on an infinite, rooted, directed $k$-ary tree where we place $k^n$ chips labeled $0,1,\dots, k^n-1$ on the root for some nonnegative integer $n$, and we say a vertex $v$ can fire if it has at least $k$ chips. When a vertex fires, we select $k$ labeled chips and send the $i$th smallest chip among them to its $i$th leftmost child. A stable configuration is reached when no vertex can fire. In this paper, we focus on stable configurations resulting from specific firing strategies based on permutations of $1, 2, \dots, n$. We then express the stable configuration as a permutation of $0,1, 2, \dots, k^n-1$ and explore its properties, such as the number of inversions and descents.
\end{abstract}

\section{Introduction}
\label{sec:in}
Chip-firing is a game on a graph in which a discrete commodity, called chips, is placed on a graph. When a vertex has sufficiently many chips, the vertex \textit{fires} and disperses a chip to each neighbor. The study of chip-firing originates from the Abelian Sandpile explored by \cite{PhysRevLett} and by \cite{dhar1999abelian} in which a stack of sand disperses when it exceeds a certain height. The study of chip-firing on graphs originates from works such as those of Anderson, Lovász, Shor, Spencer, Tardos, Winograd (\cite{zbMATH04135751}), and Björner, Lovász, Shor (\cite{MR1120415}). Since those works were published, numerous versions of the chip-firing game, e.g., the chip-firing game $M$-matrices (\cite{MR3311336}) and invertible matrices (\cite{MR3504984}), to name a few, have been studied. Chip-firing has been connected to numerous other exciting areas of study, such as the study of critical groups (\cite{MR1676732}), matroids (\cite{MR2179643}), and potential theory (\cite{MR2971705}).

\subsection{Unlabeled chip-firing on directed graphs}
\label{sec:unlabeledchipfiring}

In the unlabeled chip-firing game on directed graphs, indistinguishable chips are put on the vertices of a directed graph $G = (V, E)$. For each vertex $v$, if $v$ has at least $\outdeg(v)$ chips, it can fire. In other words, when a vertex fires, it sends one chip to each out-neighbor and loses $\outdeg(v)$ chips. When the graph reaches a state in which no vertex can fire, the distribution of chips over the vertices is a \textit{stable configuration}. We formally define key chip-firing terminology in Section~\ref{sec:definitions}.

\begin{example}
Figure~\ref{fig:exampleunlabel} shows the unlabeled chip-firing process with $4$ indistinguishable chips initially placed at the root of an infinite, directed binary tree.
\end{example}

\begin{figure}[htbp]
\centering
    \subfigure[\centering Initial configuration with $4$ chips]{{\includegraphics[width=0.4\linewidth]{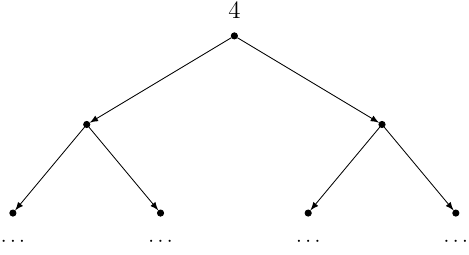} }}%
    \qquad
    \subfigure[\centering Configuration after firing once]{{\includegraphics[width=0.4\linewidth]{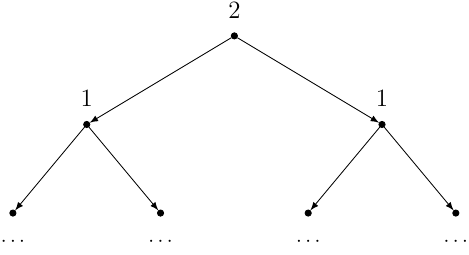} }}%
    \qquad
    \subfigure[\centering Configuration after firing twice]{{\includegraphics[width=0.4\linewidth]{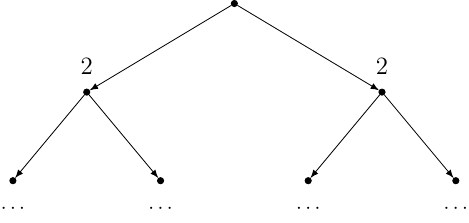} }}%
    \qquad
    \subfigure[\centering Stable configuration]{{\includegraphics[width=0.4\linewidth]{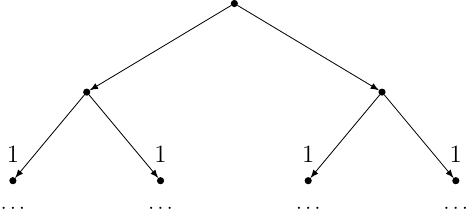} }}%
    \caption{Example of unlabeled chip-firing on an infinite directed, rooted binary tree}%
    \label{fig:exampleunlabel}
\end{figure}

A \textit{configuration} $\mathcal{C}$ is a distribution of chips over the vertices of a graph. A key property of unlabeled chip-firing on directed graphs is the following global confluence property, which is analogous to the global confluence of chip-firing on undirected graphs (c.f., Theorem 2.2.2 of \cite{klivans2018mathematics}) and stabilization of the Abelian Avalanche model of \cite{ MR1215018, gabrielov1994asymmetric}.

\begin{theorem}[Theorem 1.1 of \cite{MR1203679}]\label{thm:GlobalConfluence}
For a directed graph $G$ and initial configuration $\mathcal{C}$ of chips on the graph, the unlabeled chip-firing game will either run forever or end after the same number of moves and at the same stable configuration. Furthermore, the number of times each vertex fires is the same regardless of the sequence of firings taken in the game.
\end{theorem}

\subsection{Labeled chip-firing on directed graphs}
\label{labeledchipfirng}

Labeled chip-firing is a type of chip-firing game in which each chip has a distinct label. Labeled chip-firing was originally invented by Hopkins, McConville, and Propp (\cite{MR3691530}) in the context of one-dimensional lattices. Much more recently, it has been studied in the context of undirected binary tree graphs by \cite{MR4827886} and by the authors of this paper (\cite{inagaki2024chipfiringundirectedbinarytrees}). In \cite{MR4887467}, the authors have introduced labeled chip-firing on directed, rooted $k$-ary trees.

In this paper, we continue the exploration from \cite{MR4887467} of labeled chip-firing on infinite directed $k$-ary trees for $k \geq 2$. In that paper, the chip-firing game is defined as follows. Consider an infinite, rooted directed $k$-ary tree, i.e., an infinite directed tree with a root in which each vertex has outdegree $k$. For each vertex, its $k$ children are ordered from left to right in ascending order. Since each vertex $v$ has $\outdeg(v) = k$, a vertex can fire when it has at least $k$ chips. When a vertex fires, we pick a set of $k$ chips on the vertex to disperse and, for each $j \in [k]$, the $j$th smallest chip in the set is sent to the $j$th leftmost child of the fired vertex.

In labeled chip-firing, the global confluence property does not hold. In other words, depending on which sets of chips are fired during the stabilization process, one can obtain different stable configurations from the same initial configuration of labeled chips.

\begin{example}
Consider a directed binary tree where each vertex has two children with $4$ labeled chips: $(0,1,2,3)$ at the root. Notice that since chips are only sent along directed edges, once a chip is sent to the left or right, it cannot go back. Therefore, if we fire the pair of chips $(0,1)$ first, we end up with a different stable configuration than if we fire the pair $(1,2)$ first. Figure~\ref{fig:confluencebreak} illustrates these initial firings. 
\end{example}

\begin{figure}[h]
\centering
    \subfigure[\centering Configuration after firing $(0,1)$]{{\includegraphics[width=0.4\linewidth]{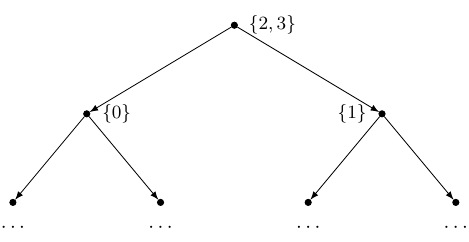} }}%
    \qquad
    \subfigure[\centering Configuration after firing $(1,2)$]{{\includegraphics[width=0.4\linewidth]{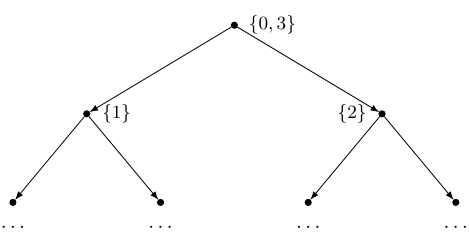} }}%
    \caption{Example of confluence breaking}%
    \label{fig:confluencebreak}
\end{figure}

As mentioned in \cite{MR4887467}, to obtain different stable configurations, we pick different sets of chips to fire. This property motivates a new branch of research in chip-firing, even in the context of relatively simple graphs like trees (e.g. \cite{MR4827886, inagaki2024chipfiringundirectedbinarytrees}).

\subsection{Motivation}
In the final section of \cite{MR4827886}, Musiker and Nguyen pose the following questions for labeled chip-firing on undirected binary trees:
\begin{itemize}
    \item What are the possible stable configurations?
    \item How many stable configurations are there?
\end{itemize}
In \cite{inagaki2024chipfiringundirectedbinarytrees}, the authors of this paper provided a partial answer to the second question by upper-bounding the number of possible stable configurations in the setting of undirected binary trees. In \cite{MR4887467}, the authors addressed the above questions, but in the context of chip-firing on the directed $k$-ary trees. 

\subsection{Our Objective and Main Definition}\label{sec:ObjectiveMainDefinition}
In this paper, we continue our work describing the possible stable configurations from chip-firing on $k$-ary trees by considering permutation-based strategies of the game, of which the ``bundling strategies" described in \cite{MR4887467} are a special case.

\begin{definition}\label{def:maindef}
For a permutation $w = w_1w_2\dots w_n$ of $1,2, \dots, n$ and chips $0,1\dots, k^n-1$ starting at the root of the $k$-ary tree, we define a \textit{permutation-based chip-firing strategy} $F_w$ to be so that when firing vertices at the $i$th layer, chips are dispersed so that the $j$th leftmost child of the fired vertex receives chips whose $k$-ary expansion has $j-1$ for its $w_i$th most significant digit.\end{definition}

\begin{example}\label{ex:F132}
    Consider the directed binary tree with $2^3$ labeled chips $0,1, \dots, 7$ initially at the root and permutation $w = 132$. As illustrated in Figure~\ref{fig:132strategy}, we use the firing strategy $F_{132}$ to stabilize the configuration of chips. Since $ w_1=1$, we fire the root vertex (i.e., the sole vertex in the first layer) so that the left child gets chips whose labels have binary expansions starting with $0$ and the right child gets chips whose labels have binary expansions starting with $1$. Then, because $w_2= 3$, we fire each vertex $v$ in the second layer so that the leftmost child of $v$ obtains chips that were on $v$ whose least significant bit in the binary expansion is $0$, and the rightmost child of $v$ gets chips from $v$ whose least significant bit is $1$. Finally, since $w_3= 2$, we observe that for each vertex $v$ in the third layer of the tree, we find that the left child of that vertex will receive chips from $v$ whose second most significant bit is $0$, and the right child will receive chips whose second most significant bit is $1$.
\end{example}

\begin{figure}[htbp]
\centering
    \subfigure[\centering Initial configuration]{{\includegraphics[width=0.35\linewidth]{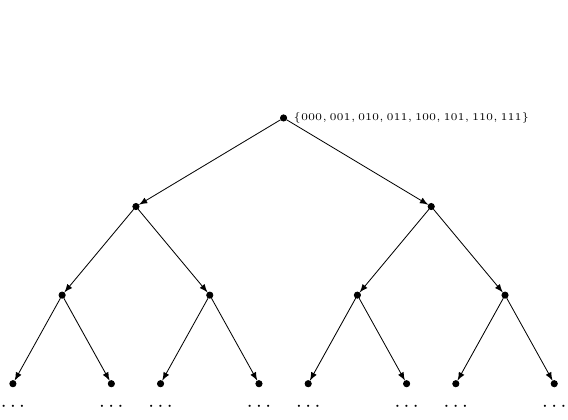} }}%
    \qquad
    \subfigure[\centering Configuration after sorting and firing by digit $1$]{{\includegraphics[width=0.35\linewidth]{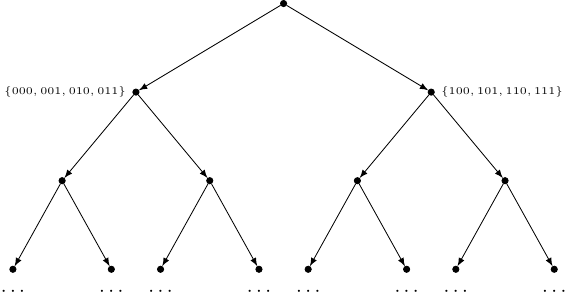} }}%
    \qquad
    \subfigure[\centering Configuration after sorting and firing by digit $3$]{{\includegraphics[width=0.35\linewidth]{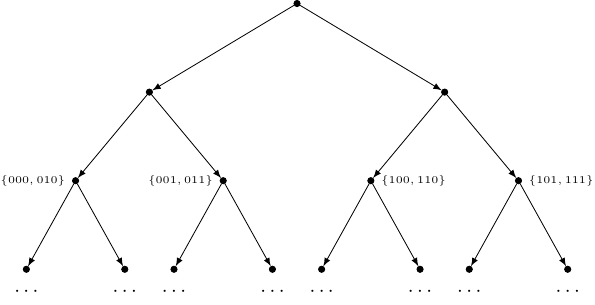} }}%
    \qquad
    \subfigure[\centering Stable configuration after sorting and firing by digit $2$]{{\includegraphics[width=0.35\linewidth]{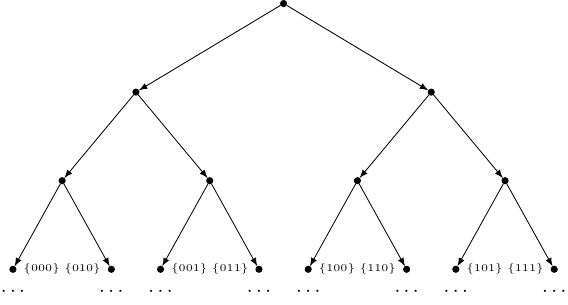} }}%
    \caption{Example of firing with strategy $F_{132}$}
    \label{fig:132strategy}
\end{figure}

\subsection{Roadmap}

In Section~\ref{sec:prelim}, we give preliminaries and definitions that are used throughout the paper. 

In Section~\ref{sec:bundling}, we formally introduce the notion of chip-firing strategies $F_w$ corresponding to permutations $w \in S_n$ that produce several interesting properties in the stable configuration. For every chip $c$ and firing strategy $F_w$, we calculate the final destination of chip $c$. We show that if permutation $w$ is lexicographically earlier than permutation $w'$, then the stable configuration corresponding to $w$ is lexicographically earlier than the one corresponding to $w'$.

In Section~\ref{sec:inversions}, we then prove a formula for the number of inversions in stable configurations resulting from $F_w$. We study properties of the set of possible numbers of inversions resulting from permutation-based strategies. In particular, we show that for a $k$-ary tree when starting with $k^n$ chips at the root, the number of inversions is always divisible by $\frac{(k-1)^2k^n}{4}$.

Afterward, in Section~\ref{sec:descents}, we show that the descent set of a stable configuration resulting from applying $F_w$ is a function of the permutation $w$. We calculate this function. We show that possible descents have to be divisible by $k$. We show a connection between the number of descents in $w$ and the configuration resulting from $F_w$ and the order of the reverses of their Lehmer codes.

In Section~\ref{sec:TailsExamples}, we explore permutations with increasing or decreasing tails.

\section{Preliminaries and Definitions}
\label{sec:prelim}
In this section, we introduce key definitions for chip-firing on directed trees, which are similar to those introduced in our previous paper \cite{MR4887467}.
\subsection{Permutations}
For natural number $n$, we define $[n] := \{1, 2, \dots, n\}$. We define the permutations $S_n$ of $[n]$ to be the reorderings of the sequence $1,2, \dots, n$.

For permutation $w$ of length $n$, we say that a pair of indices $a, b \in [n]$ is an \textit{inversion} if both $w_a < w_b$ and $b < a$. We say that an index $i\in [n-1]$ is a \textit{descent} of $w$ if $w_i > w_{i+1}$. For $w \in S_n$, the \textit{descent set} of $w$ is the set of all descents of $w$.

For permutation $w=w_1w_2\dots w_n$, a \textit{decreasing tail} is the longest suffix of permutation $w_rw_{r+1}\dots w_n$ such that it is decreasing: $w_r > w_{r+1} > \dots > w_n$ and either $r=1$ or $w_{r-1} < w_r$. We similarly define an \textit{increasing tail} of $w$.

\subsection{The underlying graph}
\label{sec:definitions}
The underlying graph for this paper is the infinite, rooted, directed $k$-ary tree. 

In a directed graph, we say that a vertex $v$ has \textit{parent} $v_p$ if there is a directed edge $v_p \to v$. If $v$ has parent $v_p$, we say that $v$ is the \textit{child} of $v_p$.  Furthermore, we say that a vertex $u$ is a \textit{descendant} of vertex $v$ if there exists a directed path from $v$ to $u$. We define a \textit{rooted tree} to be a directed graph in which every vertex, except for a designated vertex called the \textit{root}, has exactly one parent. An \textit{infinite directed k-ary tree} is an infinite directed rooted tree where each vertex has outdegree $k$.  A vertex $v$ is on layer $\ell+1$ if $v$ can be reached from the root via a directed path traversing $\ell$ vertices. In our convention, the root $r$ is on layer $1$.

\subsection{Unlabeled chip-firing on directed \texorpdfstring{$k$}{k}-ary trees}
\label{sec:unlabeled}

In our setting, we denote the \textit{initial configuration} of chip-firing as a placement of $N$ chips on the root. A vertex $v$ can \textit{fire} if and only if it has at least $\outdeg(v)=k$ chips. When vertex $v$ \textit{fires}, it transfers a chip from itself to each of its $k$ children. We define a \textit{stable configuration} to be a placement of chips over the vertices such that no vertex can fire.

Consider unlabeled chip-firing on infinite directed $k$-ary trees when starting with $k^n$ chips at the root, where $n \in \mathbb{N}^+$. As the stable configuration and the number of firings do not depend on the order of firings, we can assume that we start from layer 1 and proceed by firing all the chips on the given layer before going to the next layer. Thus, for each $i \in [n]$, each vertex on layer $i$ fires $k^{n-i}$ times and sends $k^{n-i}$ chips to each of its children. In the stable configuration, each vertex on layer $n + 1$ has exactly $1$ chip, and for all $i \neq n + 1$, the vertices on layer $i$ have $0$ chips.

\subsection{Labeled chip-firing on directed \texorpdfstring{$k$}{k}-ary trees}
\label{sec:labeled}
In the setting of labeled chip-firing, when a vertex fires, it chooses and fires $k$ of its chips so that among those $k$ chips, the one with the $i$th smallest label gets sent to the $i$th leftmost child from the left. A \textit{strategy} is a procedure dictating an order in which $k$-tuples of chips on a vertex get fired from which vertex.

In this paper, we assume $k \geq 2$ since if $k =1$ and the initial configuration has a nonzero number of chips, then the chip-firing process does not stop. Also, we always start with $k^n$ chips at the root, since in this case, the chip-firing game has desirable symmetries.

We assume that our chips are labeled by numbers 0 through $k^n-1$. In this setting, we can represent chip labels through their $k$-ary expansion, i.e., a $k$-ary string of length $n$ that can start with zeros.

We write each stable configuration as a permutation of $0, 1, 2, \dots, k^{n}-1$, which is the sequence of chips in the $(n+1)$st layer of the tree in the stable configuration, read from left to right. For instance, the stable configuration in Figure~\ref{fig:directedex} would be denoted by the permutation $0, 1, 2, 3$.

We now give an example of a labeled chip-firing game on the directed $k$-ary tree for $k=2$.
\begin{example}
Consider a directed binary tree with $4$ labeled chips: $(0,1,2,3)$ at the root. Figure~\ref{fig:directedex} shows a possible sequence of firings leading to a stable configuration.
\end{example}
\begin{figure}[htbp]
    \centering
    \subfigure[\centering Initial configuration with $4$ chips]{{\includegraphics[width=0.4\linewidth]{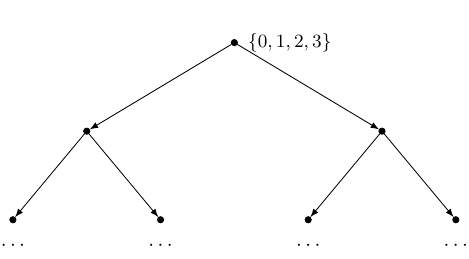} }}
    \qquad%
    \subfigure[\centering Configuration after firing root once]{{\includegraphics[width=0.4\linewidth]{fourchipex22} }}%
    \qquad
     \subfigure[\centering Configuration after firing root a second time]{{\includegraphics[width=0.4\linewidth]{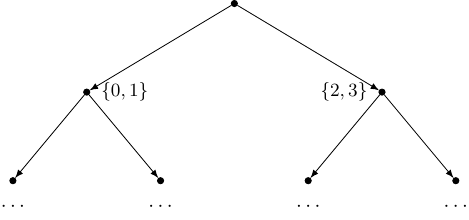} }}%
    \qquad
    \subfigure[\centering Stable configuration]
    {{\includegraphics[width=0.4\linewidth]{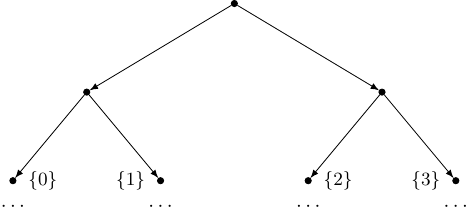} }}%
    \caption{Example of labeled chip-firing in a directed binary tree with $4$ chips}
    \label{fig:directedex}
\end{figure}

In the previous example, observe that any order in which the same pairs of chips are fired from the same vertex yields the same distribution of chips to the children. This is a fact that holds in general: in chip-firing on directed $k$-ary trees, given that a vertex fires a set of $k$-element tuples of labeled chips, any order in which those $k$-tuples of chips are fired yields the same distribution of chips to the children.

\section{Permutation-Based Strategies of Chip-Firing}
\label{sec:bundling}

In this section, we begin our discussion on stable configurations on $k$-ary trees resulting from permutation-based firing strategies (Definition~\ref{def:maindef}). Recall that for permutation $w \in S_n$, the \textit{permutation-based firing strategy} $F_w$ is the following procedure: for each $i \in [n]$, vertex $v$ in layer $i$, and $j \in \{0, 1, 2, \dots, k-1\}$, fire each vertex so that all chips in $v$ that have $j$ as the $w_i$th digit from the left are transferred to the $(j+1)$th child of $v$. We use $\mathcal{C}_{k, n, w}$ to denote the stable configuration resulting from applying the strategy $F_w$ to a $k$-ary tree with $k^n$ labeled chips $0,1\dots, k^n-1$. We express the stable configuration as the sequence of chips in the $n+1$ first layer read from left to right, which is a permutation of $0, 1, \dots, k^n-1$.

\begin{example}
Suppose we have a binary tree starting with $2^3$ labeled chips and a permutation $w = 132 \in S_3$. Example~\ref{ex:F132} and Figure~\ref{fig:132strategy} illustrates the firing strategy $F_{132}$. We write the corresponding stable configuration $\mathcal{C}_{2,3,132}$ as 0, 2, 1, 3, 4, 6, 5, 7.
\end{example}

\begin{definition}
Each vertex $v$ on the layer $\ell$ can be defined by a string of length $\ell-1$ of integers 1 through $k$, where the $i$th term is $j$ if the path from the root to $v$ passes through the $j$th child on layer $i+1$. We call this string a \textit{traversal string}. For the traversal string $t$ of length $i-1$ of integers $1$ through $k$, we use $v_t$ to denote the vertex defined by the string $t$.
\end{definition}

Given a permutation-based strategy $F_w$ and the label of a chip, one can calculate where the chip traverses during the chip-firing game and where it lands in the stable configuration.

\begin{proposition}\label{prop:GeneralUnbundlingOperation}
Consider a firing strategy $F_w$, corresponding to the permutation $w \in S_n$, and let $t=t_1t_2\dots t_{n'}$ be a traversal string whose terms are in $\{1, \dots, k\}$ for $n' \in [n]$. Then the set of chips that arrive at $v_t$ is exactly the chips $c$ with $k$-ary expansion $c=a_1a_2\dots a_{n'}$ such that $a_{w_i} = t_i-1$ for each $i \in [n']$.

In particular, for $n'=n$, the set of chips on $v_t$ in the stable configuration consists of only the chip with k-ary expansion $a_1a_2\dots a_{k}$ such that $a_{w_i}=t_i-1$ for each $i \in [n]$.
\end{proposition}
\begin{proof}
    It follows from the definition of the strategy $F_w$.
\end{proof}

The above proposition helps us to calculate the number of possible final positions of a chip when the stabilization strategy of the tree corresponds to a permutation.
\begin{proposition} Consider a $k$-ary directed tree starting with labeled chips $0,1 \dots, k^n-1$ at the root. Let $c=a_1a_2\dots a_{n}$ (written in $k$-ary expansion) be a chip, and let $f_i$ be the frequency at which $i$ appears in the expansion.

Given that the $k$-ary tree is stabilized via a strategy corresponding to some permutation, there are $\frac{n!}{f_0! f_1! \cdots f_{k-1}!}$ possible positions of chip $c$ in the stable configuration. 
\end{proposition}
\begin{proof}
    From Proposition~\ref{prop:GeneralUnbundlingOperation}, we know that for fixed chip $c=a_1a_2\dots a_n$, if strategy $F_w$ corresponds to the permutation $w=w_1w_2\dots w_n$, chip $c$ ends up in vertex $v_t$ where $t = (a_{w_1}+1)(a_{w_2}+1)\dots(a_{w_n}+1)$. We find that the range of this correspondence is given by the set of all permutations of $(a_{1}+1),(a_{2}+1),\dots,(a_{n}+1)$. Since each term in the expansion of $c$ is in $0,1, \dots, k-1$ and since the frequency of $i \in \{0,1, \dots, k-1\}$ in the expansion is $f_i$, we find that there are $\frac{n!}{f_0!f_1!\dots f_{n-1}!}$ possible values of $v_t$, the vertex in which chip $c$ ends up when firing is done via a permutation-based strategy.
\end{proof}

\begin{example}
    The second vertex in the stable configuration corresponds to the traversal string 111\ldots 1112 of length $n$. Suppose the firing strategy is described by permutation $w$, then the $k$-ary expansion of the chip that arrives at the second vertex is $c=a_1a_2\dots a_{n}$, where $a_{w_n} = 1$, and the other $a_i$ are zero. Thus, the chip is $k^{n-w_n}$.
\end{example}

There is a natural lexicographic order on permutations as well as on stable configurations. For $u, w$ that are either both permutations or both stable configurations, we use $u < w$ to denote that $u$ is lexicographically earlier than $w$.

Consider the following bijection on permutations in $S_n$: Take a permutation $w \in S_n$. Reverse the permutation, then subtract each term from $n+1$. We denote this bijection by $B$. For string $w$, we use $\rev(w)$ to stand for the reversal of $w$.

\begin{example}
    If we start with permutation 1243, reversing it gets 3421, and subtracting from 5555 gets 2134. Thus, $B(1243) = 2134$.
\end{example}

\begin{remark}
    Permutation $u$ being lexicographically earlier than $w$ does not imply that $B(u)$ is later or earlier than $B(w)$. For instance, compare two permutations $2341 < 3214$ in order. We find that $B(2341) = 4123 > B(3214) = 1432$; thus, the order is reversed. On the other hand, $2341 < 4321$ and $B(2341) = 4123 < B(4321) = 4321$, and the order is not reversed.
\end{remark}

\begin{theorem}\label{thm:LexicoGraphicOrder}
    Let $k \geq 2$ and $n\in \mathbb{N}$. If $w$ and $w'$ are in $S_n$, then the stable configuration $\mathcal{C}_{k, n, w}$ corresponding to $w$ is lexicographically earlier than the stable configuration $\mathcal{C}_{k, n, w'}$ corresponding to $w'$ if and only if $B(w) < B(w')$.
\end{theorem}

\begin{proof}
Write $w =w_1w_2\dots w_n$ and $w'= w_1'w_2'\dots w_n'$. Let $i_0$ be the first index at which $B(w)_{i_0} \neq B(w')_{i_0}$. Correspondingly, $n+1-i_0$ is the largest index $j$ such that $w_{j} \neq w_j'$. Note that if $B(w') > B(w)$, then $B(w')_{i_0} > B(w)_{i_0}$, implying that $w_{n+1-i_0} > w'_{n+1-i_0}$. Likewise, if $B(w') < B(w)$, then $B(w')_{i_0} < B(w)_{i_0}$ and consequently $w_{n+1-i_0} < w'_{n+1-i_0}$.

Let $v$ denote the leftmost vertex in the layer $n+2-i_0$. After applying strategy $F_w$, vertex $v$ receives chips that have zeros in places $\{w_1, w_2, \dots, w_{n+1-i_0}\}$. Similarly, after applying strategy $F_w'$, vertex $v$ receives chips that have zeros in places $\{w_1', w_2', \dots, w_{n+1-i_0}'\}$. Since $n+1-i_0$ is the last index such that $w_j\neq w_j'$, we know that the sets $\{w_1, w_2, \dots, w_{n+1-i_0}\}$ and $\{w_1', w_2', \dots, w_{n+1-i_0}'\}$ are the same. That means vertex $v$ receives the same chips for both strategies. Both strategies starting from layer $n+2-i_0$ and below are identical, implying that the stable configurations of the subtree rooted at $v$ are the same for both strategies.

Let $v$ denote the second vertex on layer $n+2-i_0$. After applying strategy $F_w$, the set of chips on $v'$ is the set of all integers in $0,1,\dots, k^n-1$ that both have $0$'s in digits $w_1,w_2, \dots, w_{n-i_0}$ and have the digit $1$ in the $w_{n-i_0+1}$th digit. Similarly, after applying strategy $F_w'$, the set of chips on $v'$ is the set of all integers in $0,1,\dots, k^n-1$ that both have $0$'s in digits $w_1,w_2, \dots, w_{n-i_0}$ and have the digit $1$ in the $w_{n-i_0+1}'$th digit. The smallest chip of the subtree rooted at $v'$ is $k^{n-w_{n+1-i_0}}$ if we follow strategy $F_w$, and is $k^{n-w_{n+1-i_0}'}$ if we followed strategy $F_w'$. These chips will end up at the leftmost child of the subtree of $v'$.

Therefore, if $B(w) < B(w')$, then $w_{n+1-i_0} > w_{n+1-i_0}'$ and $k^{n-w_{n+1-i_0}} < k^{n-w_{n+1-i_0}'}$, implying that the stable configuration corresponding to $w'$ is lexicographically later than that of $w$. The case of $B(w) > B(w')$ is similar.
\end{proof}

\begin{example}
Consider a binary tree with $2^4$ chips. The stable configuration $\mathcal{C}_{2, 4, 3214}$ for $3214$ starts with $0, 1, 8, 9$, and the stable configuration $\mathcal{C}_{2, 4, 2341}$ for $2341$ starts with $0, 8, 1, 9$. Meanwhile, $B(3214) = 1432$ and $B(2341) = 4123$. We see that $B(2341)$ is lexicographically later than $B(3214)$, and the corresponding stable configuration is later too.
\end{example}

As $B$ is a bijection, we deduce that distinct permutations create distinct stable configurations, as we state in the corollary below.

\begin{corollary}
  Consider a $k$-ary tree with labeled chips $0, 1, \dots, k^n-1$ starting at the root. Let $w, w' \in S_n$. The configurations resulting from $w$ and $w'$ are distinct if and only if $w \neq w'$.
\end{corollary}

We answer the following key question throughout the rest of this paper: If we know some information about the permutation $w \in S_n$, what can we say about the stable configuration? To begin to answer this question, we have the following result.

\begin{proposition}
If $F_w$ is our firing strategy, then the permutation pattern $w$ appears in the terminal configuration.
\end{proposition}

\begin{proof}
Let $w = w_1w_2\ldots w_n$. We find the following subsequence in the terminal configuration:
$$\sum_{i \in [n]\setminus \{w_1\}}(k-1)k^{n-i},\ \sum_{i \in [n]\setminus \{w_2\}}(k-1)k^{n-i},\ \dots,\ \sum_{i \in [n]\setminus \{w_n\}}(k-1)k^{n-i}.$$
To see that this is a subsequence, we know from Proposition~\ref{prop:GeneralUnbundlingOperation} that $\sum_{i \in [n]\setminus \{w_1\}}k^{n-i}$ is in the subtree rooted at $v_{1}$ (leftmost child of root vertex.), chip $\sum_{i \in [n]\setminus \{w_2\}}k^{n-i}$ is in the subtree rooted at $v_{k1}$, chip $\sum_{i \in [n]\setminus \{w_3\}}k^{n-i}$ is in the subtree rooted at $v_{kk1}$, and so on. This sequence is in the same relative order as $w_1, w_2, \dots, w_n$.
\end{proof}

\section{Inversions in Resulting Configurations}
\label{sec:inversions}

In this section, we count the number of inversions in the stable configurations resulting from $F_w$ in general, not just the maximal case studied in \cite{MR4887467}.

\subsection{The number of inversions}

We now compute the number of inversions in the terminal configuration resulting from the firing strategy $F_w$ being applied to the $k$-ary tree. For this we will use the \textit{Lehmer code} of the permutation $w=w_1w_2\dots w_n$: $$c_w = ((c_{w})_1, (c_{w})_2, \dots, (c_{w})_n),$$
where $(c_{w})_i$ is the number of terms in $w$ that are right of $w_i$ that are less than $w_i$. It is well-known that mapping permutations in $S_n$ to their Lehmer codes is a bijective correspondence between permutations in $S_n$ and $\{0, 1, \dots, n-1\}\times \{0, 1, \dots, n-2\} \times \dots \times \{0\}$. From the definition of Lehmer code, we obtain the key observation that for any $i \in [n]$, the term $w_i$ is the $((c_w)_i+1)$th smallest term in $w_iw_{i+1}\dots w_n$. One property of the Lehmer code is that it preserves ordering, i.e., if $w$ is before $w'$ lexicographically, then $c_w$ is lexicographically before $c_{w'}$. To see this, one can observe that for $w < w'$, there is minimal $i_0$ such that $w_{i_0} < w_{i_{0}}'$ and consequently $(c_{w})_{i_0} < (c_{w'})_{i_0}$ while $(c_w)_{i} = (c_{w'})_{i}$ for all $i < i_0$.

\begin{example}\label{ex:LehmerCode}
For permutation 45312, the Lehmer code is 33200.
\end{example}

The Lehmer code is a good tool for expressing the number of inversions in the terminal configuration. We denote the number of inversions in $\mathcal{C}_{k, n, w}$, the configuration resulting from applying strategy $F_w$ on a $k$-ary tree with $k^n$ labeled chips starting at the root, as $I(k, n, w)$.

\begin{theorem}\label{thm:Lehmer}
Let $k \geq 2$, $n \in \mathbb{N}$, and $w$ be any permutation in $S_n$. Let $c_w$ be the Lehmer code of permutation $w$. Then the number of inversions in $\mathcal{C}_{k, n, w}$ is
\[I(k, n, w) = \binom{k}{2} \sum_{i=1}^{n}\binom{k^{(c_w)_{i}}}{2}k^{2n-i-2(c_w)_i-1} = \frac{k^n(k-1)}{4}\sum_{i=1}^n(k^{n-i} -k^{n-i-(c_w)_i}).\]
\end{theorem}

\begin{proof}
Let $w = w_1w_2\dots w_n$. First, we count the inversions in $\mathcal{C}_{k, n, w}$ that result from two chips that end up in subtrees rooted at different children of the root vertex. Consider that when sorting chips by the $w_1$th most significant digit in the $k$-ary representation, we know that the $j$th child from the left contains the set of chips
$$S_j = \left\{k^{n-w_1}(j-1) + \sum_{m \in \{0,1, \dots, n-1\}\setminus \{n-w_1\}}d_m k^m: d_m \in \{0,1, \dots, k-1\} \right\}.$$
Therefore, for each $j', j \in [k]$ such that $j' > j$, we find that for each chip $x = k^{n-w_1}(j-1)+\sum_{m \in \{0, 1, \dots, n-1\}\setminus \{n-w_1\}}d_mk^m$ in the $j$th child of the root, the set of chips in the $j'$th child of the root that are less than $x$ is
$$\left\{k^{n-w_1}(j-1) + \sum_{m \in \{0,1, \dots, n-1\}\setminus \{n-a_1\} }d_m'k^m: \sum_{m = n-w_1+1}^{n-1}d_m'k^{m} < \sum_{m = n-w_1+1}^{n-1} d_mk^{m} \right\}.$$ The cardinality of this set is $k^{n-w_1}\sum_{m=0}^{w_1-2}d_{m+n-w_1+1}k^m $. Therefore, the number of inversions in $C_{k, n, w}$ that result from a chip sent to the $j'$th child of the root and another sent to the $j$th child of the root is
$$\binom{k^{c_i}}{2}k^{2n-2w_1}.$$
Therefore, the total number of inversions that result from two chips ending up in the subtrees rooted at different children of the root is
$$\sum_{j=1}^{k-1}(k-i)\binom{k^{w_1-1}}{2} (k^{n-w_1})^2.$$
Since $(c_w)_1 = w_1-1$ by definition of Lehmer code of $w$, we find that this quantity is equal to
$$\sum_{j=1}^{k-1}(k-i)\binom{k^{(c_w)_1} }{2} (k^{n-(c_w)_1-1})^2.$$

Note that after the root cannot fire, we find that there are $k^{n-1}$ chips on each root's child. When ignoring the $w_1$th most significant digit, taking relative order $w'$ of $w_2w_3\dots w_n$, the firing of the $k^{n-1}$ chips on each child of the root will yield the same number of inversions as firing $k^{n-1}$ using strategy $w'$. The total number of inversions obtained from two chips that are in the same subtree of the root is $I(k, n-1, w')$. Thus,
\begin{equation}\label{eq:RecursiveRelation}
I(k, n, w) = \sum_{i=1}^{k-1}(k-i)\binom{k^{(c_w)_1} }{ 2} (k^{n-(c_w)_1-1})^2 + k I(k, n-1, w').
\end{equation}

We now prove that for fixed $k \geq 2$, for all $n \in \mathbb{N}$ and $w \in S_n$, we have
\[I(k, n, w) = \binom{k}{2} \sum_{i=1}^{n}\binom{k^{(c_w)_{i}} }{ 2}k^{2n-i-2(c_w)_i-1}\]
via induction on $n$. For our base case, $n=0$, we know that $\mathcal{C}_{k, 0, w}$ only consists of one chip. Hence $I(k,0, w) = 0$. Now consider the inductive step. 

We shall show that for any $k \geq 2$ and $w \in S_{n_0+1}$, we have
\[I(k, n_0+1, w) = \binom{k}{2} \sum_{i=1}^{n_0+1} \binom{k^{(c_w)_i}}{2}k^{2n_0+2-i-2(c_w)_i-1}\]
assuming that for any $w' \in S_{n_0}$ we have $$I(k, n_0, w') = \binom{k}{2} \sum_{i=1}^{n_0}\binom{k^{(c_{w'})_{i}} }{ 2}k^{2n-i-2(c_{w'})_i-1}.$$
By Eq.~(\ref{eq:RecursiveRelation}), we have $I(k, n+1, w) = \binom{k}{2}\binom{k^{(c_w)_1}}{2}k^{2n+2-c_1-1} + k I(k, n_0, w')$, where $w' \in S_{n_0} $ is the result of taking the relative order of $w_2w_3\dots w_{n_0+1}$. Since $(c_{w'})_{j} = (c_{w})_{j+1}$ for each $j \in [n_0]$, the inductive hypothesis tells us that \begin{equation*}
    \begin{split}
        I(k, n+1, w) &= \binom{k}{2}\binom{k^{(c_w)_1}}{2}k^{2n+2-(c_w)_1-1} + k I(k, n_0, w') \\ &= \binom{k}{2}\binom{k^{c_1}}{2}k^{2n+2-c_1-1} +k \binom{k}{2}\sum_{i=1}^{n_0} \binom{k^{(c_{w})_{i+1}}}{2} k^{2n_0 - i-2(c_w)_{i+1}-1} \\ &= \binom{k}{2}\sum_{i=1}^{n_0+1}\binom{k^{(c_w)_i}}{2}k^{2n_0+1-i-2(c_w)_{i}}.
    \end{split}
\end{equation*} This proves the inductive hypothesis and hence completes the proof of the first formula.
For the second formula, we have
\begin{equation*}
\begin{split}
I(k, n, w) &= \binom{k}{2}\sum_{i=1}^n\binom{k^{(c_w)_i}}{2}k^{2n-i-2(c_w)_i-1} = \frac{k(k-1)}{4}\sum_{i=1}^n (k^{2(c_w)_i}-k^{(c_w)_i})k^{2n-i-2(c_w)_i-1} \\
&= \frac{k-1}{4}\sum_{i=1}^n (k^{2n-i}-k^{2n-i-(c_w)_i}) = \frac{k^n(k-1)}{4}\sum_{i=1}^n(k^{n-i} -k^{n-i-(c_w)_i}).
\end{split}
\end{equation*}\end{proof}

\begin{example}
For an identity permutation, the Lehmer code consists of all zeros. Thus, the corresponding number of inversions is 0.
    Suppose $w=132$ with Lehmer code 010. The corresponding number of inversions is $\binom{k}{2} k \binom{k}{2}$. If $k=2$, we obtain that the corresponding stable configuration has 2 inversions.
\end{example}

\begin{example}
    We list all permutations in $w \in S_3$ along with corresponding values of $I(2, 3, w)$, the number of inversions in $\mathcal{C}_{2, 3, w}$, in Table~\ref{tab:DecsentsInversions}.
\end{example}

\begin{example}\label{ex:Radixk}
A radix-$k$ digit reversal permutation $R_k'(n)$ is a specific permutation of $k^n$ numbers from $0$ to $k^n-1$ that we now describe. We represent each number from $0$ to $k^n -1$ in base $k$ and prepend it with zeros, so each number becomes a string of length $n$. Then we reverse each string and interpret it as an integer. The resulting sequence is $R_k'(n)$. For example, $R_2'(3)$ is $04261537$, which in binary is $000$, $100$, $010$, $110$, $001$, $101$, $011$, $111$. 

One can observe that, for $w \in S_n$ defined by $w_i= n+1-i$, the permutation induced by configuration $\mathcal{C}_{k, n, w}$ is the same as $R_k'(n)$. This permutation has $I(k, n, w) = \frac{k^{2n}-nk^{n+1} + (n-1)k^n}{4}
$ inversions. Theorem 6.1 of \cite{MR4887467} states that $\mathcal{C}_{k, n, w}$ has the maximal number of inversions.
\end{example}

\begin{example}
    For $n=2$, there are two permutations giving two possible numbers of inversions: 0 for permutation 12, and, for permutation 21, it is 
    \[\frac{k^{4}-2k^{3+1}+k^2}{4} = \frac{k^{2}(k-1)^2}{4}.\]
\end{example}

We will show later that number of inversions $I(k,n,w)$ is always divisible by $\frac{k^{2}(k-1)^2}{4}$.

\subsection{Properties of the number of inversions}

The smallest number of inversions for given $k$ and $n$ corresponds to the identity permutation, which is lexicographically the earliest. The largest number of inversions corresponds to the permutation, which is the reversal of identity and is lexicographically the latest. One might wonder whether lexicographically earlier permutations imply fewer inversions. This is not the case. For example, for $k=2$, the earlier permutation 2413 generates 28 inversions in the terminal configuration, while the later permutation 3124 generates only 24 inversions. However, there is a partial order that is preserved by the map $I(k, n, \cdot): S_n \to \mathbb{N}$ for fixed $k$ and $n$. 

For $n$ entry vectors $\vec{x}$ and $\vec{y}$, we define an entry-wise order $\preceqdot$ so that $\vec{x} \preceqdot \vec{y}$ if and only if $x_i \leq y_i$ for all $i \in [n]$. We use this order for Lehmer codes. When $\vec{x} \preceqdot \vec{y}$, we say that $\vec{x}$ is \textit{dominated by} $\vec{y}$. Similarly, we say that the permutation $w$ \textit{is dominated by} $w'$, denoted by $w \preceq w'$, if and only if $c_w \preceqdot c_{w'}$.

We show that the map $I(k, n, \cdot): S_n \to \mathbb{N}$ respects the domination order.

\begin{theorem}
\label{thm:inversionorder}
Let $k \geq 2$ and $n \in \mathbb{N}$. If $w, w' \in S_n$ and $w \preceq w'$, then the number $I(k, n, w)$ of inversions in the stable configuration $\mathcal{C}_{k, n, w}$ corresponding to $w$ is less than or equal to the number $I(k, n, w')$ of inversions in the stable configuration $\mathcal{C}_{k, n, w'}$ corresponding to $w'$.
\end{theorem}

\begin{proof}
By definition of $\preceq$ we find that $w \preceq w'$ implies $c_{w} \preceqdot c_{w'}$, which implies $(c_{w})_i \leq (c_{w'})_i$.

From Theorem~\ref{thm:Lehmer}, we know that for any $k, n \in \mathbb{N}$ and $w \in S_n$,
\begin{equation*}
I(k, n, w) = \frac{k^n(k-1)}{4}\sum_{i=1}^n(k^{n-i} -k^{n-i-(c_w)_i}).
\end{equation*}
If $(c_{w})_i \leq (c_{w'})_i$, then $-k^{n-i-(c_w)_i} \leq -k^{n-i-(c_w')_i}$. Thus, if $w \preceq w'$, then we obtain $I(k, w, n) \leq I(k, w', n)$.\end{proof}

\begin{corollary}
\label{cor:CorollaryNumInveresions}
    The number of inversions $I(k,n,w)$ is an integer multiple of $\frac{(k-1)^2k^n}{4}$.
\end{corollary}

\begin{proof}
    We have a coefficient $\frac{k^n(k-1)}{4}$ in front of the summation. In addition, each term $k^{n-i} -k^{n-i-(c_w)_i}$ under the summation is divisible by $k-1$. The corollary follows.
\end{proof}
For $k = 2$, the possible numbers of inversions are especially straightforward to describe.

\begin{proposition}\label{prop:k2caseInversionsArithmeticProgression}
    Fix $n \in \mathbb{N}$, the set of all possible values of $I(2, n, w)$ is $$\{I(2, n, w): w \in S_n\} = \{2^{n-2} m: m \in \{0, 1,2, \dots, 2^n-1-n\}\}.$$ In particular, the set of all possible values of $I(2, n, w)$ is an arithmetic progression of step size $2^{n-2}$.
\end{proposition}

\begin{proof}
We first observe that Corollary~\ref{cor:CorollaryNumInveresions} applied to the case of $k=2$ implies that $I(2, n, w) \in \{2^{n-2}m: m \in \mathbb{N}\}$. We also know from Example~\ref{ex:Radixk} that the largest possible number of inversions is $2^{n-2}(2^n-1-n)$. Therefore, we find that $\{I(2, n, w): w \in S_n\} \subseteq \{2^{n-2} m: m \in \{0, 1,2, \dots, 2^n-1-n\}\}$.

Now we need to show that any number in the range $[2^n-1-n]$ can appear in the expression $\sum_{i=1}^n (2^{n-i} - 2^{n-i-(c_w)_i})$.
We see that $\sum_{i=1}^n (2^{n-i} - 2^{n-i-(c_w)_i}) = 2^n-1 - \sum_{i=1}^n 2^{n-i-(c_w)_i}$. Thus, it is enough to show that we can get any number between $n$ and $2^n-1$ in the expression $\sum_{i=1}^n 2^{n-i-(c_w)_i}$.

Consider the number $d$ in the range from $n$ to $2^n-1$. We can represent $d$ in binary as a sum of distinct powers of 2. Unless $d=2^n-1$, the number of such powers is less than $n$. Now, we can pick one power and split it into two halves. As $d > n$, we can continue this process until we get $n$ powers of 2. This will give us our representation.
\end{proof}

\subsection{The set of possible numbers of inversions resulting from \texorpdfstring{$F_w$}{}}

Now we discuss how the possible number of inversions is connected for different $n$.

\begin{proposition}\label{prop:kAResult}
If $A$ is a possible number of inversions in a stable configuration in a $k$-ary tree starting with $k^n$ chips and using permutation-based strategy $F_w$ for some $w \in S_n$, then there exists a firing strategy and a stable configuration in a $k$-ary tree starting with $k^{n+1}$ chips with $kA$ inversions.
\end{proposition}

\begin{proof}
    Suppose $A$ is the number of inversions in a terminal configuration achieved with strategy $F_w$, corresponding to permutation $w$. Consider a permutation $w'$ of $n+1$ elements that starts with 1, and the rest is order isomorphic to $w$. After firings at the root according to $F_{w'}$, all chips at the $i$th child of the root are smaller than the chips at the $j$th child as long as $i < j$. Thus, in the terminal configuration, there are no inversions between two chips that are descendants of different children of the root. On the other hand, among the chips that are on the subtree of a particular root's child, there are exactly $A$ inversions.
\end{proof}

\begin{example}
    For $k=3$ and $n=3$, the possible number of inversions is 0, 27, 81, 108, and 135. Multiplying by 3, we get 0, 81, 243, 324, and 405. At the same time, the possible number of inversions for $k=3$ and $n=4$ is 0, 81, 243, 324, 405, 729, 810, 972, 1053, 1134, 1215, 1296, 1377, 1458.
\end{example}

Let $A(k,n)$ be the increasing sequence of all possible values of $I(k, n, w)$ in stable configurations in a $k$-ary tree starting with $k^n$ chips.

\begin{proposition}
    Let $k \geq 2$. The largest element of $A(k, n)$ is smaller than the smallest element of $\{a \in \frac{A(k, n+1)}{k}: a \not\in A(k, n)\}$.
\end{proposition}
\begin{proof}
    We first remark that the set of terms in $k A(k, n)$ is equal to the set $\{I(k,n+1, w'): w' \in S_{n+1}, (c_w')_1 = 0.\}$. This is by repeating the argument from the proof of the Proposition~\ref{prop:kAResult}. This implies $\{a \in A(k, n+1): a \not\in k A(k, n)\} = \{I(k,n+1, w'): w' \in S_{n+1}, (c_{w'})_1 \neq 0.\}$.
    
    Thus to show that the largest element in $A(k,n)$ is smaller than the smallest element of $\{a \in A(k, n+1): a \not\in k A(k, n) \}$, it suffices to show that
    \[\max \{I(k, n+1, w'): w' \in S_{n+1}, (c_{w'})_1 = 0\} < \min \{I(k,n+1, w'): w' \in S_{n+1}, (c_{w'})_1 \neq 0\}.\]
    Consider a permutation $u \in S_{n+1}$ such that $(c_u)_1=0$. Then, the maximum number of inversions for the corresponding terminal configuration is 
    \[ \frac{k^{n+1} (k-1)}{4} (k^{n-1}+k^{n-2} + \dots + 1 - n) = \frac{k^{n+1}(k-1)}{4}\left(\frac{k^n-1}{k-1}-n\right).\]
    On the other hand, consider permutation $u \in S_{n+1}$ such that $(c_u)_1 \neq 0$. Then, the minimum number of inversions for the corresponding terminal configuration is 
    \[\frac{k^{n+1}(k-1)}{4} (k^{n} - k^{n-1}) = \frac{k^{n+1}(k-1)}{4}(k-1)k^{n-1},\]
    which is larger than the previous value, completing the proof.
\end{proof}

 Consider a sequence $A'(k,n) = \frac{4A(k,n)}{k^{n}(k-1)^2}$. We showed that this is an integer sequence in Corollary~\ref{cor:CorollaryNumInveresions}. We also showed that the sequence $A'(k,n)$ forms a prefix to the sequence $A'(k,n+1)$. It follows that there exists a limiting sequence $A'_\infty(k)$, such that $A'(k,n)$ is its prefix for any $n$.

\begin{example}
    Sequence $A'_\infty(3)$ (A381462 in \cite{oeis}) starts as 0, 1, 3, 4, 5, 9, 10, 12, 13, 14, 15, 16, 17, 18. Sequence $A'_{\infty}(4)$ (A381463 in \cite{oeis}) starts as 0, 1, 4, 5, 6, 16, 17, 20, 21, 22, 24, 25, 26, 27.
 \end{example}



\section{Counting Descents in Resulting Configurations}
\label{sec:descents}

In this section, we study the descent set of $\mathcal{C}_{k, n, w}$. The descents depend on the \textit{support} of the Lehmer code $c_w$, denoted by $\supp(c)$. The support is the set of elements $j \in [n]$ such that $c_j > 0$. Suppose $c_i = 0$. It means that in the corresponding permutation $w$, the element $w_i$ is smaller than all consecutive elements: $w_i < w_j$, for $j > i$. Such elements $w_i$ are called \textit{right-to-left minima} of the permutation $w$. Thus, the support is the set of indices of the permutation elements that are not right-to-left minima.

We show that a number belongs to the descent set of $\mathcal{C}_{k, n, w}$ if and only if in its $k$-ary presentation using $n$ digits, the position of the last non-zero digit belongs to the support of $c_w$.

\begin{theorem}\label{thm:DescentSets}
        Let $k \geq 2$ and $n \in \mathbb{N}$. Let $w \in S_n$. Then, the descent set of $\mathcal{C}_{k, n, w}$ is the set of all
    \begin{equation}\label{eq:DescentSet}
    \bigcup_{s \in \supp(c_w)}\left\{ \sum_{i=1}^{s}d_ik^{n-i}: d_s \in [k-1] \quad \mbox{ and } \quad \forall i \in [s-1], d_i \in \{0, 1, \dots, k-1\}\right\}.
            \end{equation}
\end{theorem}

\begin{proof}
    Consider a vertex $v$ at the layer $n$ that corresponds to a descent, and let $c$ be the chip that ends in $v$. Vertex $v$ must be the right-most child of its parent. Let $v'$ be the next vertex with chip $c'$, and let $p$ be the closest common ancestor of $v$ and $v'$ on layer $j$. That means $v$ and $v'$ traverse through $p$, and firing at $p$ sends $v$ to child $i$ of $p$ and send $v'$ to child $i+1$. Chip $c$ is the largest chip that visits the $i$th child of $p$, and chip $c'$ is the smallest chip that visits the $(i+1)$st child.

    It follows that $c < c'$ if and only if the sequence of firings starting from layer $j$ down is in increasing order of the digit place that is used for firing. That means vertex $v$ does not correspond to a descent if and only if the permutation $w$ that describes the firing order has an increasing tail starting from place $j$. That means the Lehmer code value at place $j$ is zero. Thus, the traversal string $t = t_1t_2\dots t_n$ that leads to the vertex $v$ has digit $t_j$ less than $k$, and $t_i = k$, for $i>j$. Thus, the counting order of the vertex $v$ when written in $k$-ary ends in a non-zero digit at the $j$th place, followed by zeros.
\end{proof}

From the above theorem, we find the set of possible descents given $k \geq 2$ and $n \in \mathbb{N}$.
\begin{corollary}
    Let $k \geq 2, n \in \mathbb{N}$. Consider a $k$-ary tree with $k^n$ chips $0,1,2 \dots, k^n-1$. The set of possible descents is $\{k, 2k, \dots, k^n-k\}$. 
\end{corollary}

\begin{proof}
    As $n \not\in \supp(c_w)$, every descent represented in $k$-ary ends in zero by Theorem~\ref{thm:DescentSets}.
\end{proof}

Also, from Theorem~\ref{thm:DescentSets}, we can derive the number of descents in stable configurations. The number of descents is a number that can be written in base $k$ using only digits 0 and $k-1$ and having not more than $n$ digits. Let $D(k, n, w)$ denote the number of descents in the stable configuration of chips $0, 1, \dots, k^n-1$ resulting from applying $F_w.$

\begin{theorem}\label{thm:DescentNumbers}
For permutation $w$ with Lehmer code $c_w$, the number of descents in $\mathcal{C}_{k, n, w}$ is
\[D(k, n, w) = \sum_{j \in \supp(c_w)} (k-1) k^{j-1}.\]
\end{theorem}

\begin{proof}
    Let $d$ be a $k$-ary string of length $n$ that may start with zeros. From Theorem~\ref{thm:DescentSets}, we know that the integer represented by $d$ belongs to the descent set if and only if its last non-zero digit is in place $j$, where $j \in \supp(c_w)$. For each $j$, the number of such strings is $(k-1)k^{j-1}$. The theorem follows.
\end{proof}

\begin{example}
\label{ex:descents2}
For $k=2$, the number of descents in a terminal configuration can be any number between 0 and $2^{n-1}-1$ inclusive.
\end{example}

\begin{example}
    For $k=3$, the number of descents in a terminal configuration can be any number that can be written in base 3 using digits 0 and 2 and having not more than $n$ digits. Integers that can be written in base 3 using digits 0 and 1 are the Stanley sequence: the lexicographically earliest sequence that does not contain 3-term arithmetic progressions. The Stanley sequence is sequence A005836 in the OEIS (\cite{oeis}). Our numbers of descents are twice the numbers in the Stanley sequence.
\end{example}

\begin{example}
    Let $k=2$. We compute the set of descents and the number of inversions for terminal configurations for all permutations of length $3$ in Table~\ref{tab:DecsentsInversions}.
    \begin{table}[htbp]
    \centering
    \begin{tabular}{|c|c|c|c|c|c|c|}
    \hline
        $w$     & $c_w$    & $\rev(\supp(c_w))$ &$\mathcal{C}_{2, 3, w}$ & \# inversions & \# descents   & Descent Set    \\ \hline
        $123$   & $000$    & 000            & $(0,1,2,3,4,5,6,7)$    & 0             & 0    & $\emptyset$  \\ 
        $132$   & $010$    & 010   & $(0,2,1,3,4,6,5,7)$             & $2$           & $2$       & $\{2, 6\}$    \\ 
        $213$   & $100$    & $001$   & $(0,1,4,5,2,3,6,7)$            & $4$           & 1       & $\{4\}$      \\ 
        $231$   & $110$    & $011$  & $(0,4,1,5,2,6,3,7)$     & 6             & 3    & $\{2, 4, 6\}$\\
        $312$   & $200$    & $001$   & $(0,2,4,6,1,3,5,7)$                & 6              & 1    & $\{4\}$    \\ 
        $321$   & $210$    & $011$      & $(0,4,2,6,1,5,3,7)$  & 8            & 3        & $\{2, 4, 6\}$   \\
         \hline 
     \end{tabular}
    \caption{Descents and inversions of terminal configurations for $k=2$ and all permutations of length $3$}
    \label{tab:DecsentsInversions}
\end{table}
As we can see, the number of descents can be any number between 0 and 3, as was predicted in Example~\ref{ex:descents2}. We also see that descent sets are the same for permutations that have the same support of their Lehmer codes.
\end{example}

Suppose $k \geq 2$ and $n \in \mathbb{N}$ are fixed. One may wonder if for $w, w' \in S_n$, for configuration $\mathcal{C}_{k, n, w}$ having fewer descents than $\mathcal{C}_{k, n, w'}$ implies $\mathcal{C}_{k, n, w}$ has fewer inversions than does $\mathcal{C}_{k, n, w'}$. It turns out that this is not the case.
\begin{example}
    Consider $\mathcal{C}_{2, 4, 1324}$ and $\mathcal{C}_{2, 4, 3124}$. Since $1324$ has Lehmer code $0100$, Theorems~\ref{thm:Lehmer} and Theorem~\ref{thm:DescentNumbers} tell us that $\mathcal{C}_{2, 4, 1324}$ has $2$ descent and $8$ inversions. Since $3124$ has Lehmer code $2000$, Theorem~\ref{thm:Lehmer} and Theorem~\ref{thm:DescentNumbers} tell us that $\mathcal{C}_{2, 4, 3124}$ has $1$ descent and $24$ inversions. 
\end{example}

One may also ask whether, similar to inversions, a lexicographically earlier support of a Lehmer code corresponds to fewer descents. This is not the case.

\begin{example}\label{ex:Permutations}
Consider permutations $34125$ and $14523$ with respective Lehmer codes 22000 and 02200. The supports of these Lehmer codes are $(1,1, 0, 0, 0)$ and $(0, 1, 1, 0, 0)$, respectively. Thus, we find that $D(k,5, 34125)=(k-1)+(k-1)k$, whereas $D(k,5,13425) = (k-1)k + (k-1)k^2$, not matching the order.
\end{example}

However, more descents do correspond to a lexicographically later reversal of the support of the Lehmer code.

\begin{theorem}
Let $w$ and $w^\prime$ be in $S_n$. Let $c_w$ and $c_{w^\prime}$ be the Lehmer codes of $w, w'$ respectively. If the number of descents for $\mathcal{C}_{k,n,w}$ is greater than that of $\mathcal{C}_{k,n,{w^\prime}}$, then $\rev(\supp(c_w))$ must be lexicographically later than $\rev(\supp(c_{w^\prime}))$.
\end{theorem}

\begin{proof}
From Theorem~\ref{thm:DescentNumbers}, we have that for any $w$ the $i$th index of $\supp(c_w)$ when expressed as a binary string contributes $(k-1)k^{i-1}$ descents if and only if $(c_w)_i \neq 0$. Thus since the number of descents for $\mathcal{C}_{k,n,w} $ is greater than that of $\mathcal{C}_{k,n,{w^\prime}}$, we have that $i_0 \in [n]$, the largest index in $[n]$ such that $(c_w)_{i_0} \neq 0$ is bigger than $i_0'$, the largest index in $[n]$ such that $(c_{w'})_{i_0'} \neq 0$. Also, $i_0$ and $i_0'$ are respectively the largest indices at which $\supp(c_w)$ and $\supp(c_{w'})$ are nonzero. Therefore, $\rev(\supp(c_{w}))$ is lexicographically later than $\rev(\supp(c_{w'}))$.
\end{proof}

\begin{example}
    Consider $34125$ and $14523$ from Example~\ref{ex:Permutations} with respective Lehmer codes 22000 and 02200. The reversals of the supports of the Lehmer codes are $00011$ and $00110$. The reversals are in the same order as the number of descents.
\end{example}

\section{Permutations with Decreasing or Increasing Tails}
\label{sec:TailsExamples}

The smallest number of descents and inversions corresponds to permutations with the all-zeros Lehmer code, which is only the identity permutation. The largest number of descents corresponds to permutations with Lehmer code that contain only one zero at the end. These are permutations that end in~1. There are $(n-1)!$ of them.

In general, permutations with many zeros at the end of their Lehmer codes have few descents. Suppose a permutation has $r$ zeros at the end of its Lehmer code. That means the last $r$ permutation values are in increasing order.

\begin{example}
We define \textit{valley permutations} to be permutations of $[n]$ that monotonically decrease until reaching a global minimum and afterwards monotonically increase. The Lehmer code $c_w$ for a valley permutation $w \in S_n$ with global minima at index $i_0$ is $c_w = (w_1-1, w_2-1, \dots, w_{i_0-1}-1, 0, 0, \dots, 0)$. It follows that the number of descents in the stable configuration corresponding to valley permutations is a function of $i_0$. More specifically, for fixed $k \geq 2$ and for valley permutation $w \in S_n$ with global minima at index $i_0$ there are exactly $\sum_{i=1}^{i_0-1} (k-1)k^{i-1}$ descents by Theorem~\ref{thm:DescentNumbers}. The number of inversions can be expressed through the values of the permutation before the minimum: $I(k, n, w) = \frac{k^n(k-1)}{4}\sum_{i=1}^n (k^{n-i} - k^{n-i-(c_w)_i}) = \frac{k^n(k-1)}{4}\sum_{i=1}^{i_0-1}(k^{n-i} - k^{n-i-w_i+1})$. 
\end{example}

Now, we describe some bounds on the number of descents and inversions in the stable configuration depending on the permutation. We start with permutations with an increasing tail.

\begin{proposition}\label{prop:IncreasingTail}
    Given permutation $w$, if there exists an $i_0\in \mathbb{N}$ such that for all indices $i\geq i_0$, we have $w_i < w_{i+1}$, then $\mathcal{C}_{k, n, w}$ has at most $\sum_{j=1}^{i_0-1}(k-1)k^{j-1}$ descents and at most $\binom{k}{2} \sum_{i=1}^{i_0-1}k^{i-1}\binom{k^{n-i}}{2}$ inversions.
\end{proposition}
\begin{proof}
    Because $w_i < w_{i+1}$ for all integers $i$ such that $i \geq i_0$, we find that $(c_{w})_i = 0$ for all $i \geq i_0$. Therefore, the support of $c_w$ is a subset of $[i_0-1]$. Thus, by Theorem~\ref{thm:DescentNumbers}, we have that there are at most $\sum_{j=1}^{i_0-1}(k-1)k^{j-1}$ descents in $\mathcal{C}_{k, n, w}$.

    Because $(c_{w})_i = 0$ for all $i \geq i_0$, we have
    $$I(k, n, w) = \binom{k}{2} \sum_{i=1}^{i_0-1}k^{i-1}\binom{k^{(c_w)_{i}} }{2}k^{2n-2i-2(c_w)_i} = \frac{k-1}{4}\sum_{i=1}^{i_0-1}(k^{2n-i} - k^{2n-i-(c_w)_i}).$$
    Since for each $i$, we have $(c_w)_i \leq n-i$ by definition of Lehmer code, we obtain
    \begin{equation*}
        \begin{split}
            I(k, n, w) & \leq   \frac{k^n(k-1)}{4} \sum_{i=1}^{i_0-1}(k^{n-i}-k^{n-i-(n-i)}) \\ & = \frac{k(k-1)}{4} \sum_{i=1}^{i_0-1}k^{n-1}(k^{n-i}-1) = \binom{k}{2} \sum_{i=1}^{i_0-1}k^{i-1}\binom{k^{n-i}}{2}.
        \end{split}\end{equation*}\end{proof}

One can observe that the upper bounds on the number of inversions and descents in Proposition~\ref{prop:IncreasingTail} are tight.
\begin{example}
Consider any positive integers $i_0, n, k$ such that $i_0 < n$ and $k \geq 2$. Let $w$ be a permutation in $S_n$ defined by $w_i = n+1-i$ for $i \in [i_0-1]$ and $w_i = i-i_0+1$ for $i \in \{i_0, i_0+1, \dots, n \}$. This is a special case of a valley permutation, where the increasing part consists of smaller numbers than the decreasing part. We obtain that the Lehmer code of this permutation is $(c_w)_j = n-i$ for each $i \in [i_0-1]$ and $(c_w)_i = 0$ for $i \in \{i+1, i+2, \dots, n\}$. Therefore we obtain from Theorem~\ref{thm:Lehmer} that the number of inversions in the stable configuration $\mathcal{C}_{k,n, w}$ resulting from firing strategy $F_w$ is $\binom{k}{2}\sum_{i=1}^{i_0-1} k^{i-1}\binom{k^{n-i}}{2}$. This is exactly the upper bound on the number of inversions in $\mathcal{C}_{k, n, w}$ for $w$ with increasing tail starting at $i_0$.

Also observe that Theorem~\ref{thm:DescentNumbers} and the fact that $\supp(c_w) = [i_0-1]$ imply that $\mathcal{C}_{k,n,w}$ has exactly $\sum_{i=1}^{i_0-1}(k-1)k^{k-1}$ descents. This is equal to the upper bound on the number of descents in $\mathcal{C}_{k,n,w}$ from Proposition~\ref{prop:IncreasingTail}. 
\end{example}

On a similar note, we calculate the lower bound for the number of inversions and descents in $\mathcal{C}_{k, n, w}$ in the case where $w$ has a decreasing tail.
\begin{proposition}\label{prop:DecreasingTail}
    If there exists an $i_0 \in \mathbb{N}$ such that for all indices $i\geq i_0$, we have $w_i > w_{i+1}$, then $\mathcal{C}_{k, n, w}$ has at least $\sum_{j=i_0}^{n-1}(k-1)k^{j-1}$ descents and at least $\frac{k^n(k-1)}{4}\left(\frac{k^{n-i_0+1}-1}{k-1} - (n - i_0 + 1)\right)$ inversions.
\end{proposition}
\begin{proof}
    Because for all $i \geq i_0$, we have $w_i > w_{i+1}$, we find that $(c_{w})_i =  n-i$ for all $i \geq i_0$. Therefore, the zeros of the support of $c_w$ form a subset of $[i_0-1] \cup \{n\}$. Thus, by Theorem~\ref{thm:DescentNumbers}, there are at least $\sum_{j=i_0}^{n-1}(k-1)k^{j-1}$ descents in $\mathcal{C}_{k, n, w}$. In addition, by Theorem~\ref{thm:Lehmer}, we obtain
    \begin{equation*}\begin{split}
        I(k, n, w) & = \frac{k^n(k-1)}{4}\sum_{i=1}^n (k^{n-i} - k^{n-i-(c_w)_i}) \\  & \geq \frac{k^n(k-1)}{4}\sum_{i=i_0}^n (k^{n-i} - k^{n-i-(c_w)_i})  \\&  = \frac{k^n(k-1)}{4}\sum_{i=i_0}^n (k^{n-i} - 1)  = \frac{k^n(k-1)}{4}\left(\frac{k^{n-i_0+1}-1}{k-1} - (n - i_0 + 1)\right). \end{split}
    \end{equation*}\end{proof}

As was the case with Proposition~\ref{prop:IncreasingTail}, we find that the bounds on the numbers of descents and inversions from Proposition~\ref{prop:DecreasingTail} are tight.

\begin{example} Consider any positive integers $i_0, n, k$ such that $i_0 < n$ and $k \geq 2$. Let $w$ be a permutation in $S_n$ defined by $w_i = i$ for $i \in [i_0-1]$ and $w_i = n+i_0-i$ for $i \in \{i_0, i_0+1, \dots, n \}$. We obtain that the Lehmer code of this permutation is defined by $(c_w)_i = n-i$ for each $i \in \{i_0, i_0+1, \dots, n\}$ and $(c_w)_i = 0$ for $i \in [i_0-1]$. Therefore we obtain from Theorem~\ref{thm:Lehmer} that in the stable configuration $\mathcal{C}_{k,n, w}$ resulting from $F_w$, there are $I(k, n, w) = \frac{k^n(k-1)}{4}\sum_{i=i_0}^n (k^{n-i} - k^{n-i-(c_w)_i}) = \frac{k^n(k-1)}{4} \sum_{i=i_0}^n (k^{n-i} - 1) = \frac{k^n(k-1)}{4}\left(\frac{k^{n-i_0+1}-1}{k-1} - (n - i_0 + 1)\right)$ inversions. This is equal to the lower bound on the number of inversions in $\mathcal{C}_{k, n, w}$ from Proposition~\ref{prop:DecreasingTail} for $w$ with decreasing tail starting at $i_0$.

Also observe that Theorem~\ref{thm:DescentNumbers} and the fact that $\supp(c_w) = \{i_0, i_0 + 1, \dots, n\}$ imply that $\mathcal{C}_{k,n,w}$ has exactly $\sum_{i=i_0}^{n-1}(k-1)k^{k-1}$ descents. This is equal to the lower bound on the number of descents in $\mathcal{C}_{k,n,w}$ from Proposition~\ref{prop:DecreasingTail}.
\end{example}

\acknowledgements
\label{sec:ack}
We thank Professor Alexander Postnikov for suggesting the topic of labeled chip-firing on directed trees and helping formulate the proposal of this research problem, and for helpful discussions.  The MIT Department of Mathematics financially supports the first and second authors. 

All figures in this paper were generated using TikZ.

\nocite{*}
\bibliographystyle{abbrvnat}
\bibliography{DMTCS_Inaga_Khov_Luo_Bibliography}
\label{sec:biblio}

\end{document}